\documentclass[11pt, reqno]{amsart}    	
\usepackage[margin=1in]{geometry}
\usepackage{setspace}
\usepackage{graphicx}	
\usepackage{subcaption}
\usepackage{float}
\usepackage[toc]{appendix}
\usepackage[english,algoruled,vlined,resetcount,linesnumbered]{
			algorithm2e}

\usepackage[backref=page]{hyperref}
\hypersetup{%
	colorlinks,
	linkcolor={red!50!black},
	citecolor={green!50!black},
	urlcolor={blue!50!black}
}
\usepackage{amsmath}%
\usepackage{}
\usepackage{amsthm}
\usepackage{amsfonts}%
\usepackage{amssymb}%
\usepackage{graphicx}
\usepackage{mathrsfs}
\usepackage{lineno}
\usepackage{bm}
\usepackage[draft]{todonotes}   
\usepackage{hyperref}
\usepackage{lineno}
\usepackage{setspace}
\usepackage{enumerate}

\usepackage{paralist}	

\usepackage{amssymb}
\usepackage{amsmath, amsthm, amssymb}

\newtheorem{thm}{Theorem}[section]

\newtheorem{prop}[thm]{Proposition}
\newtheorem{claim}[thm]{Claim}
\newtheorem{lemma}[thm]{Lemma}
\newtheorem{cor}[thm]{Corollary}

\newtheorem{fact}{Fact}

\theoremstyle{definition}
\newtheorem{defi}[thm]{Definition}
\def\eps{\varepsilon}

\usepackage{tikz}
\usepackage{tikz} \usetikzlibrary{shapes,arrows}

\newcommand*{\rom}[1]{\expandafter{\romannumeral #1\relax}}

\usepackage{amsmath}
\numberwithin{equation}{section}
\begin{document}
\title[]{Minimum degree conditions for Hamilton $l$-cycles in $ k $-uniform hypergraphs}
\author{Jie Han}
\address{School of Mathematics and Statistics and Centre for Applied Mathematics, Beijing Institute of Technology, Beijing, China}
\email{\tt han.jie@bit.edu.cn}

\author{Lin Sun}
\address{School of Mathematics and Statistics, Qingdao University, Qingdao 266071, PR China. School of Mathematics, Shandong University,
	Jinan, China.}
\email{linsun77@163.com}

\author{Guanghui Wang}
\address{School of Mathematics, Shandong University, Jinan, China.}
\email{ghwang@sdu.edu.cn}	

\begin{abstract}
We show that for $ \eta>0 $ and sufficiently large $ n $, every 5-graph on $ n $ vertices with $\delta_{2}(H)\ge (91/216+\eta)\binom{n}{3}$ contains a  Hamilton 2-cycle.
This minimum 2-degree condition is asymptotically best possible.
Moreover, we give some related results on  Hamilton $ \ell $-cycles with $ d $-degree for  $\ell\le d \le k-1$  and  $1\le \ell < k/2$.
\end{abstract}	
\maketitle	
\section{Introduction}
The study of Hamilton cycles is an important topic in graph theory and extremal combinatorics. 
Dirac's classic result \cite{1952Some} states that every graph whose minimum degree is at least as large as half the size of the vertex set contains a Hamiltonian cycle. 
In recent years, extending Dirac’s theorem to hypergraphs has attracted a great deal of attention.
Given $ k\ge2 $, a \emph{$ k $-uniform hypergraph} $H$ (in short, \emph{$ k $-graph}) consists of a vertex set $V$ and an edge set $E\subseteq \binom{V}{k}$, where every edge is a $ k $-element subset of $ V $. 
We denote by  $ e(H):=|E| $ the numbers of edges in $H$.
Given a $ k $-graph $H=(V,E)$ and a vertex set $S\in \binom{V}{d} $, we define $ N(S) $ to be the family of $T\in\binom{V}{k-d} $ such that $ T\cup S\in E $ and $ \deg_H(S):=|N(S)| $. 
The \emph{minimum $ d $-degree} of $H$ denoted by $ \delta_{d}(H) $ is the minimum of $ \deg_{H}(S) $ over all $ d $-element vertex sets $S$ in $H$.
For $ 1\le \ell <k $, we say that a $ k $-graph is an \emph{$ \ell $-cycle} if there exists a cyclic ordering of its vertices such that every edge consists of $ k $ consecutive vertices and two  consecutive edges (in the natural order of the edges) share exactly $ \ell $ vertices. 
If the ordering is not cyclic, we call it an \emph{$ \ell $-path} and we say the first and last $ \ell $ vertices are the ends of the path. 
By the \emph{length} of an $ \ell $-path, we mean the number of edges contained in it.
In $ k $-graphs, a $ (k-1) $-cycle is often called a \emph{tight cycle} and a $ (k-1) $-path is often called a \emph{tight path}.
A $ k $-graph on $ n $ vertices contains a \emph{Hamilton $ \ell $-cycle}  if it contains an $ \ell $-cycle as a spanning subhypergraph. 
Note that a  Hamilton $ \ell $-cycle of a $ k $-graph on $ n $ vertices contains exactly $ n/(k-\ell) $ edges, implying that $(k-\ell)\mid n$.
For $ (k-\ell)\mid n $ and $ 1\le d\le k-1 $, we define the \emph{Dirac threshold} $ h_{d}^{\ell}(k,n) $ to be the smallest integer $ h $ such that every	$ n $-vertex $ k $-graph $ H $ satisfying $ \delta_{d}(H)\ge h $ contains a Hamilton $ \ell $-cycle. 
Let  \[h_{d}^{\ell}(k):=\limsup_{n\to \infty}h_{d}^{\ell}(k,n)/\binom{n}{k-d}. \]

Dirac thresholds of hypergraphs were first investigated by Katona and Kierstead \cite{2006Hamiltonian}, who showed that $ 1/2 \le h_{k-1}^{k-1}(k) \le 1-1/(2k) $ and conjectured  $ h_{k-1}^{k-1}(k) =1/2$, which was confirmed by R\"odl, Ruci\'nski, and Szemer\'edi \cite{VOJTECH2006A,2008An}, that is, for $ \eps>0 $ and large $ n $, $\delta_{k-1}(H)\ge n/2+\eps n$ guarantees a Hamilton tight cycle. 
When $(k-\ell)\mid k$, a tight cycle on $ V $ trivially contains an $ \ell $-cycle on $ V $.
So the asymptotic Dirac threshold $ h_{k-1}^{\ell}(k)=1/2$  follows as a consequence of $ h_{k-1}^{k-1}(k) =1/2$ and a construction of Markstr\"om and Ruci\'nski \cite{Markstr2011Perfect}. 
When $(k-\ell)\nmid k$, there have been a series of works on  the threshold $ h_{k-1}^{\ell}(k) $.
We collect these results in the following theorem.
\begin{thm}\cite{H2010Dirac,2010Loose,MR2652102,VOJTECH2006A,2008An} 
For any $k> \ell \ge 1$, we have	
\[h_{k-1}^\ell(k)= 
\begin{cases}
	1/2 & {(k-\ell) \mid k}\\
	\frac{1}{\lceil \frac{k}{k-\ell}\rceil (k-\ell)} & {(k-\ell) \nmid k}.
\end{cases}
\]
\end{thm}

For sufficiently large $ n $, some exact thresholds $ h_{k-1}^{\ell}(k,n) $ are known: for $ k =3$ and $\ell= 2$~\cite{2011Dirac} and for $ k \ge 3  $ and $ 1\le \ell < k/2 $~\cite{CZYGRINOW2014TIGHT,Jie2015Minimum}. 
For $d=k-2$, Bu\ss, H\`an and Schacht \cite{Bu2013Minimum} showed that $h_{2}^{1}(3)=\frac{7}{16}$. 
Han and Zhao \cite{Han2015Minimum} showed the exact result for $ h_{2}^{1}(3,n) $.
Bastos, Mota, Schacht, Schnitzer and Schulenburg \cite{Bastos2016Loose} determined $ h_{k-2}^\ell(k) = 1-\left(1-\frac{1}{2(k-\ell)} \right)^{2} $ for $ k\ge4 $, $1\le \ell<k/2$ and got the exact result in \cite{J2017Loose},  which generalizes the previous results for 3-graphs.
For tight cycles, which might be considered as the most difficult and interesting case, Polcyn, Reiher, R\"odl, Ruci\'nski, Schacht and Szemer\'edi \cite{Polcyn2020Minimum,Reiher2019Minimum}  showed the asymptotic Dirac threshold $h_{1}^{2}(3)=h_{2}^{3}(4)=5/9$. 
Lang and Sanhueza-Matamala	 \cite{2020Minimum} proved that $ h_{k-2}^{k-1}(k)=5/9$ for all $ k \ge 3 $ (the same	result was also proved independently by Polcyn,  Reiher, R\"odl, and Sch\"ulke \cite{2020On}) and also provided a general upper bound of $ 1-1/(2(k-d)) $ for $ h_{d}^{k-1}(k)$, narrowing the gap to the lower bound of $ 1-1/\sqrt{k-d}$ due to Han and Zhao \cite{2016Forbidding}.
For $ d\le k-3 $, much less is known under $ d $-degree conditions. 
Recently, H\`an, Han and Zhao \cite{han2021minimum} determined the exact value of $ h_{d}^{k/2}(k,n) $ for any even integer $ k \ge 6 $, integer $ d $ such that  $ k/2 \le d \le k-1 $ and sufficiently $ n $. 
Gan, Han, Sun and Wang \cite{large} determined the following Dirac thresholds for $1\le\ell < k/2$.

\begin{thm} [\cite{large}]\label{thmhh}
	Suppose that $k\ge3$, $k-\ell\le d < 2\ell \le k-1$ such that $2k-2\ell \ge (2(2k-2\ell-d)^2+1)(k-d-1)+1$ or suppose that $ k $ is odd, $ k\ge7,\ell=(k-1)/2 $  and  $d=k-3$, then 
	\[
	h_{d}^{\ell}(k){\tiny } = 1-\left(1-\frac{1}{2(k-\ell)}\right)^{k-d}.
	\] 
\end{thm} 

The following ``space-barrier'' construction shows that the minimum degree conditions in the aforementioned results in \cite{Bastos2016Loose,Bu2013Minimum,large,2010Loose,MR2652102} are best possible, namely,
\begin{equation}
\label{eq:lowerbd}
h_{d}^{\ell}(k){\tiny } \ge 1-\left(1-\frac{1}{2(k-\ell)}\right)^{k-d}.
\end{equation}
Given $k\ge3$, $1\leq \ell<k/2$, $ 1\leq d\le k-1 $ with $(k-\ell)\mid n$,
let $H_{k,\ell}:=(V,E)$ be an $ n $-vertex $ k $-graph such that  $ E $ consists of all $ k $-sets that intersect $ A\subseteq V $, where $ |A|=\lceil \frac{n}{2(k-\ell)}\rceil-1 $. 
Note that an $\ell$-cycle  on $ n $ vertices contains $ n/(k-\ell) $ edges and each vertex is contained in at most two edges of any $\ell$-cycle for $ \ell<k/2 $. 
So $ H_{k,\ell} $ contains no Hamilton $ \ell $-cycle and \[\delta_d(H_{k,\ell})=\binom{n-d}{k-d}-\binom{n-|A|-d}{k-d}=\left(1-\left(1-\frac{1}{2(k-\ell)}\right)^{k-d}-o(1)\right) \binom{n}{k-d}.\] 
Letting $n$ tend to infinity,  we obtain~\eqref{eq:lowerbd}.

We expect more Dirac thresholds taking the value of the space-barrier (as in Theorem~\ref{thmhh}) for all $\ell < k/2$. 
However, this is not true for large $\ell$.
Indeed, as noticed in \cite{2016Forbidding}, for $ k -\ell=o(\sqrt{k-d})) $, we have $ h_d^\ell(k) \to 1 $ as $ (k-d )\to \infty $ regardless of $ (k-\ell) \mid k $ or not.

\subsection{Our results: $ h_2^2(5) $}
In this paper we study more thresholds of Hamilton $ \ell $-cycles and focus on the case $\ell<k/2$. 
The following theorem gives the Dirac threshold for $ k=5 $, $ \ell=(k-1)/2=2 $  and  $d=k-3=2 $.
\begin{thm}\label{52}
\[  h^{2}_{2}(5)=\frac{91}{216}.	\]
\end{thm} 

Our proof in~\cite{large} fails for determining $h^{2}_{2}(5)$.
The key reason is that in~\cite{large} we could only prove Theorem~\ref{thm1} for $\alpha < 1/7$ (and one can see later in Section 3 that we need the case $\alpha=1/6$ in proving Theorem~\ref{52}).
Moreover, we also derive a connecting lemma (Lemma~\ref{conlem}) for $d=\ell$ from the Kruskal--Katona theorem -- in~\cite{large} we use a connecting lemma from~\cite{Non-linear} which works for $d>\ell$ only.
At last, we also need a better absorbing path lemma because the one we established in~\cite{large} only works for $d\ge k-\ell$ and $\ell < k/2$.

Given two $k$-graphs $F$ and $H$, an \emph{$ F $-tiling} in $ H $ is a subgraph of $H$ consisting of vertex-disjoint copies of $F$. 
The number of copies of $F$ is called the \emph{size} of the $F$-tiling.
When $F$ is a single edge, an $F$-tiling is known as a \emph{matching}. 
For $k>b\ge 0$, let $Y_{k,b}$ be the $k$-graph consisting of two edges that intersect in exactly $b$ vertices. 
For any $ 0<\varepsilon<\eta $, $k\ge3$, $1\leq \ell<k/2$ and $ 1\le d \le k-1 $, let $ t_{d}^{\ell}(k,n,\varepsilon) $ denote
the minimum $ t $ such that every $k$-graph $H$ on $n$ vertices with $\delta_d(H)\ge t$ contains a $Y_{k,2\ell}$-tiling covering all but at most $ \varepsilon n $ vertices of $ H $.  Let \[ t(k,d,\ell):=\limsup_{\varepsilon\to 0} \limsup_{n\to \infty} \frac{t_{d}^{\ell}(k,n,\varepsilon)}{\binom{n}{k-d}}.\] 
The aforementioned spare barrier construction indeed shows that
\begin{equation}
\label{eq:tlowerbd}
t(k,d,\ell) \ge 1-\left(1-\frac{1}{2(k-\ell)}\right)^{k-d}.
\end{equation}

Using the new connecting lemma and absorbing lemma mentioned above, we show the following two results for general Dirac thresholds for $\ell$-cycles with $ \ell <k/2$.
\begin{thm} \label{thmh2}
Suppose that $\eta>0$, $\ell< d \le k-1$  and  $1\le \ell <k/2$, then  there exists $ n_0 $ such that every $ k $-graph $H$ on $n\ge n_0 $  vertices with  $(k-\ell)\mid n$ and 
\[ 	\delta_{d}(H)\ge\left(\max\{t(k,d,\ell),1/3\}+\eta\right)\binom{n}{k-d} 	\]
contains a Hamilton $ \ell $-cycle,
that is, $ 	h_{d}^{\ell}(k)\le\max\{t(k,d,\ell),1/3\} $.
\end{thm}

\begin{thm} \label{thmh}
Suppose that $\eta>0$ and  $2\le \ell <k/2$, then  there exists $ n_0 $ such that every $ k $-graph $H$ on $n\ge n_0 $  vertices with  $(k-\ell)\mid n$ and 
\[ 	\delta_{\ell}(H)\ge\left(\max\left\{(1/2)^{\frac{k-\ell}{\ell}},t(k,\ell,\ell)\right\}+\eta\right)\binom{n}{k-\ell} 	\]
contains a Hamilton $ \ell $-cycle,
that is, $ 	h_{\ell}^{\ell}(k)\le\max\left\{(1/2)^{\frac{k-\ell}{\ell}},t(k,\ell,\ell)\right\} $.
\end{thm}

We do need the minimum $d$-degree condition to be at least $(\frac13 + o(1))\binom{n}{k-d}$ in both theorems above.
The reason that we do not see the constant $1/3$ in Theorem~\ref{thmh} is that $t(k,\ell,\ell) > 1 - 1/\sqrt{e} > 0.39$ by~\eqref{eq:tlowerbd}.
Similarly, when $d>\ell$ and $t(k,d,\ell) > 1/3$, we know that $ 	h_{d}^{\ell}(k) = t(k,d,\ell)$, that is, the value of $h_{d}^{\ell}(k)$ is determined solely by the tiling threshold.
\begin{cor} \label{thmh3}
If  $0.82(k-\ell)\le k-d< k-\ell$, then $ 	h_{d}^{\ell}(k) = t(k,d,\ell)$.
\end{cor}
\begin{proof}We have
\[	t(k,d,\ell) \ge 1-\left(1-\frac{1}{2(k-\ell)}\right)^{k-d}>1-e^{-\frac{k-d}{2(k-\ell)}}\ge1-e^{-0.41}>1/3,\] 
where we used  the inequality $ (1-1/n)^{n}<1/e $ for integer $ n>0 $.
So $ 	h_{d}^{\ell}(k) = t(k,d,\ell)$ by Theorem~\ref{thmh2}.
\end{proof}

\subsection{New proof ideas}
Now we briefly talk about our proof ideas.
Theorem~\ref{thmh2} and Theorem~\ref{thmh} are proved by using the absorbing method, popularized by R\"odl, Ruci\'nski, and Szemer\'edi in \cite{VOJTECH2006A}. 
The proof is divided into the following lemmas: the connecting lemma, the absorbing path lemma, the path cover lemma and the reservoir lemma.
Roughly speaking, the absorbing path lemma reduces the task of finding a Hamilton $ \ell $-cycle to the much easier problem of finding an $ \ell $-cycle  covering the majority of vertices. 
Furthermore, we compute the value of $ t(5,2,2) $ (see Theorem~\ref{yth}), which together with Theorem~\ref{thmh} implies Theorem~\ref{52}.

To prove our absorbing lemma Lemma~\ref{abslem}, we combine the swapping-absorbing idea of Reiher, R\"odl, Ruci\'nski, Schacht and Szemer\'edi \cite{Reiher2019Minimum} and the lattice-based absorbing method of the first author~\cite{han2017decision}.
Roughly speaking, by the swapping-absorbing idea, one can build the absorbers in two steps. In the first step, we show that there are many ``end absorbers'' from each of which we can ``free'' a $(k-\ell)$-set of vertices.
In the second step, it is shown that every $(k-\ell)$-set can be ``swapped'' with the free vertex set in the end absorbers, namely, one can find two short $\ell$-paths that include either of the sets as interior vertices (e.g., in the graph case, the pair $(a,b)$ is a swapper for $u$ and $v$ if both $aub$ and $avb$ form paths of length two). 
It is easy to see that one can ``concatenate'' the swappers to form longer swapper chains, and we use the reachability arguments and the lattice-based absorbing method to control the swappings. 

The rest of this paper is organized as follows. 
In Section 2 we give some preparatory results.
We use the absorbing method to prove Theorem~\ref{thmh2} and Theorem~\ref{thmh}  in Section 3.
Finally, we give proofs of  Theorem~\ref{52} in Appendix.

\section{Preliminaries}
One important ingredient of our swapping-absorbing method is the following notion of reachability introduced by Lo and Markstr\"om \cite{MR3338027}.
Given a constant $ \beta>0 $, an integer $ i \ge 1 $ and a $ k $-graph $ H $ on $ n $ vertices, we say that two vertices $ u, v $ in $ H $ are $ (\beta, i) $-reachable  if there are at least $ \beta n^{(2k-\ell)i-1} $ $ ((2k-\ell)i-1) $-sets $ T $ such that there exist vertex-disjoint $ \ell $-paths $ P_1,\dots,P_i $ of length two with $ V(P_1\cup\dots\cup P_i )=T\cup\{u\} $, and  vertex-disjoint $ \ell $-paths $ P_1',\dots,P_i' $ of length two with  $ V(P_1'\cup\dots\cup P_i' )=T\cup\{v\} $, where $ P_j $ and $ P_j' $ have the same ends for all $ j\in[i] $.
Moreover, we call $ T $ a reachable set for $ \{u, v\}$. 
Given a vertex set $ U\subseteq V(H) $, $ U  $ is said to be $ (\beta, i) $-closed if every two vertices in $ U $ are $ (\beta, i) $-reachable in $ H $. 

The following simple results will be useful.
\begin{fact}\label{mind}
Let $1\le d'\le d <k$ and $H$ be a $k$-graph on $n$ vertices. 
If $\delta_{d}(H) \ge x\binom{n-d}{k-d}$ for some $ 0\le x \le 1 $, then $\delta_{d'}(H) \ge x\binom{n-d'}{k-d'}$.
\end{fact}
\begin{proof}
Since  $\delta_{d}(H) \ge x\binom{n-d}{k-d}$, we get $\delta_{d'}(H) \ge \binom{n-d'}{d-d'} x\binom{n-d}{k-d}/ \binom{k-d'}{d-d'}\ge x\binom{n-d'}{k-d'} $.
\end{proof}                                     
We use by now a common hierarchical notation, writing $x\ll y$ to mean that there exists a function $f$, whenever $x\leq f(y)$, the subsequent statement holds. 
While multiple constants appear in a hierarchy, they are chosen from right to left.
\begin{prop}\label{c} 
Let $ q\ge2, $ $1/n\ll \beta\ll \eta$, $1\le\ell, d <k$ and $H$ be a $k$-graph on $n$ vertices with $\delta_{d}(H) \ge (1/q+\eta)\binom{n}{k-d}$.
Then every $ q $ vertices of $ V(H) $ contains two vertices which are $ (\beta, 1) $-reachable in $ H $.
\end{prop}
\begin{proof}
Take $ \beta\ll \varepsilon\ll \eta $.
By  $\delta_{d}(H) \ge (1/q+\eta)\binom{n}{k-d}$ and Fact~\ref{mind}, we get $\delta_{1}(H)\ge(1/q+\eta/2)\binom{n}{k-1} $.
For any $ q $-tuple $ v_1,v_2,\dots,v_q $ of $ V(H) $, we have $ \sum_{i=1}^{q}|N_{H}(v_i)|\ge(1+q\eta/2)\binom{n}{k-1} $.  
By the pigeonhole principle, there exist $ v_i,v_j $ such that $ |N_{H}(v_i)\cap N_{H}(v_j)|\ge \varepsilon n^{k-1} $.
Let $ P^{*} $ be a $ k $-graph which is an $ \ell $-path of length two. 
Then the link $ (k-1) $-graph $ P^{**} $ of a vertex of degree two in $ P^{*} $ is $ (k-1) $-partite. 
By the supersaturation result (see \cite{Onextremal}), we can find $ (2k-\ell-1)!\beta n^{2k-\ell-1} $ copies of $ P^{**} $ in $ N_{H}(v_i)\cap N_{H}(v_j) $. 
Given any such copy of $ P^{**} $ whose vertex set is denoted by $ T $,  we get that both $ T\cup \{v_i\} $ and $ T\cup \{v_j\} $ form a copy of $ P^{*} $. 
Overall there are at least $ \beta n^{2k-\ell-1} $ choices of $ (2k-\ell-1) $-sets for $ T $.
So $  v_i$ and  $v_j $ are $ (\beta, 1) $-reachable in $ H $.
\end{proof}

\section{Hamilton $\ell$-cycles}
In this section, we prove Theorem~\ref{thmh2} and Theorem~\ref{thmh}
\subsection{Connecting Lemmas}
We first present two versions of connecting lemma, both of which state that in any sufficiently large $ k $-graph with large minimum $ d $-degree, we can connect  any two disjoint ordered  $ \ell $-sets of vertices by a short $ \ell $-path.
When $d>\ell$, we use the following connecting lemma from \cite{Non-linear}. 
The case $ d=k-1 $ was proved earlier in~\cite{MR2652102}. 
\begin{lemma}[Connecting lemma, \cite{Non-linear}, Lemma 4.1]
\label{con lem}
Suppose that $k\ge3$ and $1\leq \ell<d \leq k-1$ such that $(k-\ell)\nmid k$, and that $1/n \ll \beta\ll \mu,1/k$.
Let $H$ be a $k$-graph on $n$ vertices satisfying $\delta_d(H)\ge \mu \binom{n}{k-d}$. Suppose  $S$ and $T$ are two disjoint ordered $\ell$-sets of $V(H)$, then there exists an $\ell$-path $P$ in $H$ with  $S$ and $T$ as ends such that $ P $ contains at most $ 8k^{5} $ vertices.	
\end{lemma}
For $ \ell < k/2 $, we derive the following connecting lemma for $ \ell $-paths from the Kruskal-Katona theorem.
For this, we first introduce the notion of (robust) shadow.
\begin{defi}
Given $ \varepsilon\ge0 $ and $ \ell\ge1 $, the \emph{$ \varepsilon $-robust $ \ell $-shadow} of a $ k $-graph $ H$, denoted by $ \partial_{\varepsilon}^{\ell}(H) \subseteq \binom{V(H)}{k-\ell} $, is the $ (k-\ell) $-graph consisting of all $ (k-\ell) $-sets lying in more than $ \varepsilon n^{\ell} $ edge in $ H $,
that is, $ \partial_{\varepsilon}^{\ell}(H)= \left\{F\in \binom{V(H)}{k-\ell}:\deg_{ H}(F)> \varepsilon n^{\ell}\right\} $.
\end{defi}
Kruskal-Katona theorem studies the size of the (0-robust) shadow of a hypergraph. 
We state the following handy version by Lov\'{a}sz \cite{bollobas1981laszlo}.
\begin{thm}[Kruskal--Katona theorem, \cite{bollobas1981laszlo}] \label{KK}
For all  integers $ 1\le \ell\le k$ and $ n_{k}\in \mathbb R$,  let $H$ be a $ k $-graph  with  at least $ \binom{n_{k} }{k} $ edges.
Then	\[ 	 |\partial_{0}^{\ell}(H)|\ge\binom{n_{k} }{k-\ell}.	\]
\end{thm}
Now we give a connecting lemma for minimum $ \ell $-degree and  $ \ell $-cycles where $1\leq \ell\le k/2$. 
The cases for $ k\ge3 $, $ 1\le\ell<k/2 $ were proved by Bu\ss, H\`an and Schacht \cite{Bu2013Minimum} and Bastos, Mota, Schacht, Schnitzer and Schulenburg \cite{Bastos2016Loose}.
\begin{lemma}[Connecting lemma]	\label{conlem}
Suppose that  $1\leq \ell\le k/2$ and $ \eta>0 $. 
Let $H$ be a $ k $-graph on $n$ vertices satisfying $\delta_{\ell}(H)\ge\left((1/2)^{\frac{k-\ell}{\ell}}+\eta\right)\binom{n}{k-\ell}$. 
Then for any two disjoint ordered $ \ell $-sets $S$ and $T$ of $V(H)$, there exists an $ \ell $-path of length two in $H$ from $S$ to $T$.	
\end{lemma}

\begin{proof}
Let $ 1/n<\varepsilon\ll \eta $.
Fix two disjoint ordered $ \ell $-sets $S$ and $T$ of $V(H)$.
In order to get the desired $ \ell $-path, it suffices to find two $ (k-\ell) $-sets $ S_1 $ and $ T_1 $ with $ |S_1\cap T_1|=\ell $ such that $ S\cup S_1\in E(H) $ and $ T\cup T_1\in E(H) $.
For $ \ell=k/2 $, we have $  \min\{|N(S)|, |N(T)|\}\ge\delta_{\ell}(H)\ge(1/2+\eta)\binom{n}{k-\ell}$. 
So there exists an $ \ell $-set $ S_1\in N(S)\cap N(T) $ such that $ S\cup S_1\in E(H) $ and $ T\cup S_1\in E(H) $.
Now we suppose $1\leq \ell< k/2$.
Consider $ N_{S} $ and $ N_{T} $ as two $ (k-\ell) $-graphs on $ V(H) $ with the edge sets $ N(S) $ and $ N(T) $ respectively.
So it suffices to show $ \partial^{k-2\ell}_{\varepsilon} (N_{S})\cap \partial^{k-2\ell}_{\varepsilon} (N_{T})\not=\emptyset $. 
Indeed, suppose $ D\in \partial^{k-2\ell}_{\varepsilon} (N_{S})\cap \partial^{k-2\ell}_{\varepsilon} (N_{T}) $, then it is easy to find $ S_1\in N_{S} $ and $ T_1\in N_{T} $ with $ S_1\cap T_1=D $.

We apply the following procedure iteratively and  get a spanning subgraph of $ N_{S} $, denoted by $ G $, which satisfies that for any $ \ell $-set $ X $ in $ V(G) $, either $ \deg_{G}(X)\ge \varepsilon n^{k-2\ell} $ or $ \deg_{G}(X)=0 $. 
If there is an $ \ell $-set $ B $ with $ \deg_{ N_{S}}(B)< \varepsilon n^{k-2\ell} $, delete all the edges containing $ B $ in $ N_{S} $. 
Note that when the process ends, the number of deleted edges is at most $ \varepsilon n^{k-2\ell}\binom{n}{\ell} $.
By the minimum $ \ell $-degree condition of $ H $, $ e(N_S)=|N(S)|\ge\left((1/2)^{\frac{k-\ell}{\ell}}+\eta\right)\binom{n}{k-\ell} $. 
So as $ \eps \ll \eta $,
\[ e(G)\ge e(N_S)- \varepsilon n^{k-2\ell}\binom{n}{\ell}>\binom{(1/2)^{\frac{1}{\ell}}n+\eps^{\frac{1}{k-\ell}}n}{k-\ell} \] and $ \partial^{k-2\ell}_{\varepsilon} (G)=\partial^{k-2\ell}_{0} (G) $.
Using Theorem~\ref{KK} to $ G $, we get $ |\partial^{k-2\ell}_{0} (G)|>\binom{(1/2)^{\frac{1}{\ell}}n+\eps^{\frac{1}{k-\ell}}n}{\ell}$.
Thus  \[|\partial^{k-2\ell}_{\varepsilon}( N_{S})|\ge |\partial^{k-2\ell}_{\varepsilon} (G)|=|\partial^{k-2\ell}_{0} (G)|>\frac{1}{2}\binom{n}{\ell}  .\]
Similar arguments show that $ |\partial^{k-2\ell}_{\varepsilon} (N_{T})|>\frac{1}{2}\binom{n}{\ell}  $.
Hence $ \partial^{k-2\ell}_{\varepsilon} (N_{S})\cap \partial^{k-2\ell}_{\varepsilon} (N_{T})\not=\emptyset $. 
\end{proof}

\subsection{Absorbing Path Lemma}
Our main contribution of this paper is to prove the following absorbing path  lemma, which gives an absorbing $\ell$-path $P$ which can absorb a small but arbitrary set of vertices.

\begin{lemma} [Absorbing path lemma]\label{abslem} 
Suppose that $k\ge3$, $1\leq \ell< k/2$ and  $1/n\ll\theta\ll\gamma \ll \eta,1/k$.
Let $H$ be a $ k $-graph on $n$ vertices satisfying  $\delta_{\ell}(H)\ge \left(\max\{1/3,(1/2)^{\frac{k-\ell}{\ell}}\}+\eta\right) \binom{n}{k-\ell}$ with $ \ell\ge2 $, or $\delta_{\ell+1}(H)\ge (1/3+\eta) \binom{n}{k-\ell-1}$.
Then there exists an $ \ell $-path $P$ with $|V(P)|\le\gamma n$ such that $ P $  can absorb any set $X\subseteq V(H)\setminus V (P)$ with $ |X|\le \theta n $, $ (k-\ell)\mid |X| $, that is, there exists an $ \ell $-path $Q$ with the same ordered ends as $P$, where $V(Q)=V(P)\cup X$. 
\end{lemma}

\begin{defi} \label{defab}
Let $H$ be a $k$-graph. Given a set $ S $ of $ k-\ell $ vertices of $ H $, we call  an ordered set  an \emph{$ S $-absorber}, if it is a sequence of $ \ell $-paths $\mathcal{Q}=(P_1,\dots,P_s)$, and there exists another sequence $\mathcal{Q'}=(P_1',\dots,P_s')$ of $ \ell $-paths such that $ V(\mathcal{Q})=V(\mathcal{Q}')\cup S $ and $ P_i,P_i' $  have the same ends for each $ i\in[s] $.
\end{defi}
It is known that if every $ (k-\ell) $-set has many absorbers, then known probabilistic arguments will produce an absorbing path.
To establish a similar property, we use a variant of the absorbing method originated from \cite{Reiher2019Minimum} and also developed in  \cite{han2017decision}.
The following example illustrates the idea of absorbers.
Given a set of $ k-\ell $ vertices $ \{v_1,\dots,v_{k-\ell}\} $, consider a set of $\ell$-paths $ P_1,\dots,P_{k-\ell} $ of length two (swappers) and a $ k $-graph $ A $ (end-absorber) containing a spanning $\ell$-path $ P_A $ with $ S'_{A}=\{w_1,\dots,w_{k-\ell}\}\subseteq V(A) $. 
For $ i\in[k-\ell] $, $ v_i $ has degree two in $ P_i $ and $V(P_i)\setminus\{v_i\} \cup\{w_i\} $ also forms an $\ell$-path in which both two edges contain $ w_i $ and with the same ends as $ P_i $.
Moreover, we require that $A- S'_{A}$ also contains a spanning $\ell$-path with the same ends as $ P_A  $. 
That is, when we absorb $ \{v_1,\dots,v_{k-\ell}\} $, $ v_i $ will play the role of $  w_i $ in $ P_i $ for $ i\in[k-\ell] $ and $ w_1,\dots,w_{k-\ell} $ will be put inside $ A- \{w_1,\dots,w_{k-\ell}\}   $. (See Figure~\ref{f3}).

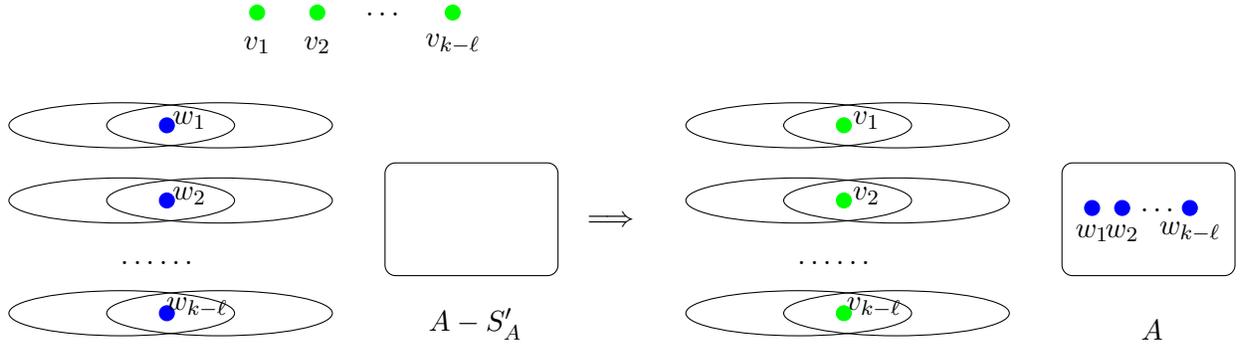
\begin{figure}[h]
\begin{center}
\begin{tikzpicture}
[inner sep=2pt,
vertex/.style={circle, draw=black!50, fill=black!50},
rect/.style={rectangle, inner sep=7,minimum size=0.8},
]			
\node at (1.8,4) [vertex,color=green] {};
\node at (2.6,4) [vertex,color=green] {};
\node at (3.5,3.8) [label=above:$\dots$] {};
\node at (4.4,4) [vertex,color=green] {};
\node at (1.8,3.3) [label=above:$v_1$] {};
\node at (2.6,3.3) [label=above:$v_2$] {};
\node at (4.4,3.3) [label=above:$v_{k-\ell}$] {};
\draw (0,0) ellipse (1.5 and 0.3);
\draw (1.3,0) ellipse (1.5 and 0.3);
\node at (0.6,0) [vertex,color=blue] {};	
\node at (1,-0.2) [label=above:$w_{k-\ell}$] {};
\draw (0,1.5) ellipse (1.5 and 0.3);
\draw (1.3,1.5) ellipse (1.5 and 0.3);
\node at (0.6,1.5) [vertex,color=blue] {};	
\node at (0.9,1.3) [label=above:$w_2$] {};
\draw (0,2.5) ellipse (1.5 and 0.3);
\draw (1.3,2.5) ellipse (1.5 and 0.3);
\node at (0.6,2.5) [vertex,color=blue] {};	
\node at (0.9,2.3) [label=above:$w_1$] {};
\node at (0.5,0.5) [label=above:$\dots$$\dots$] {};
\draw[rounded corners] (3.5,0.5) rectangle (5.8,2);
\node at (4.7,-0.5) [label=above:$ A-S'_{A}  $] {};
\node at (6.5,1) [label=above:$ \Longrightarrow  $] {};			
\draw (9,0) ellipse (1.5 and 0.3);
\draw (10.3,0) ellipse (1.5 and 0.3);
\node at (9.6,0) [vertex,color=green] {};	
\node at (10,-0.2) [label=above:$v_{k-\ell}$] {};
\draw (9,1.5) ellipse (1.5 and 0.3);
\draw (10.3,1.5) ellipse (1.5 and 0.3);
\node at (9.6,1.5) [vertex,color=green] {};	
\node at (9.9,1.3) [label=above:$v_2$] {};
\draw (9,2.5) ellipse (1.5 and 0.3);
\draw (10.3,2.5) ellipse (1.5 and 0.3);
\node at (9.6,2.5) [vertex,color=green] {};	
\node at (9.9,2.3) [label=above:$v_1$] {};
\node at (9.5,0.5) [label=above:$\dots$$\dots$] {};
\draw[rounded corners] (12.5,0.5) rectangle (14.8,2);
\node at (13.7,-0.5) [label=above:$ A $] {};
\node at (12.9,1.4) [vertex,color=blue] {};	
\node at (13.3,1.4) [vertex,color=blue] {};
\node at (14.2,1.4) [vertex,color=blue] {};
\node at (13.8,1.2) [label=above:$\dots$] {};
\node at (12.9,0.8) [label=above:$w_1$] {};
\node at (13.3,0.8) [label=above:$w_2$] {};
\node at (14.2,0.8) [label=above:$w_{k-\ell}$] {};	
\end{tikzpicture}		
\caption{a $\{v_1,\dots,v_{k-\ell}\}$-absorber, where $ S'_{A}= \{w_1,\dots,w_{k-\ell}\} $.}	
\label{f3}	
\end{center}
\end{figure}	

Our actual absorbers are a little bit more complicated, namely, we allow swapper chains of constant length and may concatenate them.
The following absorbing path was constructed and used in \cite{MR2652102}, which we will use as an end-absorber in our proof.
\begin{prop}[\cite{MR2652102}, Proposition 6.1]\label{Prop1} 
For all integers $k\ge3$ and  $1\leq \ell \leq k-1$ such that $(k-\ell)\nmid k$, there is a $k$-partite $k$-graph $A(k,\ell)$ with the following properties.
\begin{enumerate}
\item $|V(A(k,\ell))|\le k^{4}$.
\item  $V(A(k,\ell))=S'\cup X$, where $ S' $ and $ X $ are disjoint and $|S'| = k-\ell $.
\item $A(k,\ell)$ contains an $ \ell $-path $P$ with vertex set 
$X$ and ordered ends $ P^{beg} $ and $ P^{end} $.
\item \label{prop1_4}$A(k,\ell)$ contains an $ \ell $-path $Q$ with vertex set $S'\cup X$ and ordered ends $ P^{beg} $ and $ P^{end} $.
\item \label{prop1_5}Each edge of $A(k,\ell)$ contains at most one vertex of $ S' $.
\item Each vertex class of $A(k,\ell)$ contains at most one vertex of $ S' $.
\end{enumerate}
\end{prop}

The lattice-based absorbing method features a partition lemma, which gives a partition of $ V(H) $ such that each part is closed. We need the following partition lemma in our context.
\begin{lemma} [\cite{Non-linear}, Lemma 5.4]\label{part} 
Suppose that integers $ c,k\ge2 $, $ 1\le\ell <k $ and $ 0<1/n \ll\beta \ll \beta'\ll \delta\ll1/c $. 
Let $H$ be a $k$-graph on $n$ vertices and every set of $ c+1 $ vertices in $ V(H) $ contains two vertices that are $ (\beta', 1) $-reachable in $ H $. 
Then there exists a partition $ \mathcal{P} $ of  $V(H)$ into $ V_1,V_2,\dots,V_{r},U $ with $ r \le c $ such that for any $ i\in[r]$, $|V_{i}| \ge (\delta-\beta')n$ and $ V_{i} $ is $ (\beta, 2^{c-1}) $-closed in $ H $, and $ 0\le |U| \le c\delta n $.
\end{lemma} 
Let $ \mathcal{P}=\{V_1,V_2,\dots,V_{r},U\} $ be a vertex partition of $ H $. 
The \emph{index vector} $ i_{\mathcal{P}}(S) \in \mathbb{Z}^r  $ of a subset $ S\subseteq V  $ with respect to $ \mathcal{P} $ is the vector whose coordinates are the sizes of the intersections of $ S $ with each  $ V_{i} $, $ i\in[r] $. 
We call a vector $ \textbf{i} \in \mathbb{Z}^r  $ an $ s $-vector if all its coordinates are non-negative and their sum equals  $ s $. 
Given  $\mu > 0 $, a $ k $-vector $ \textbf{v} $ is called a $ \mu $-robust edge-vector if there are at least $ \mu n^{k} $ edges $ e $  in $ H $ satisfying $ i_{\mathcal{P}} (e) = \textbf{v} $. 
Let $ I^{\mu}_{\mathcal{P}} (H) $ be the set of all $ \mu $-robust edge-vectors.

Now we are ready to prove  Lemma~\ref{abslem}. 
The proof follows the scheme of the absorbing method and uses Lemma~\ref{con lem} and Lemma~\ref{conlem} in the obvious way. 
The additional work comes from the fact that not all the  $ (k-\ell) $-sets  have many absorbers. 
To address this we use Lemma~\ref{part}  to find a partition of $ V (H) $ into at most three parts, and classify the $ (k-\ell) $-sets that do have many absorbers. 
Then we show that we can always partition the leftover vertices together with a reserved set $ R_1 $ into $ (k-\ell) $-sets  that have many absorbers in the absorbing path.
\begin{proof}[Proof of Lemma~\ref{abslem}]
Suppose we have the constants satisfying the following hierarchy
\[1/n\ll\theta\ll\beta\ll \beta'\ll\mu\ll\delta\ll\gamma \ll \eta,1/k.\]
Let $H$ be a $k$-graph on $n$ vertices such that 
\[\delta_{\ell+1}(H)\ge (1/3+\eta) \binom{n}{k-\ell-1} \ \ \text{or}  \ \ \delta_{\ell}(H)\ge \left(\max\left\{1/3,(1/2)^{\frac{k-\ell}{\ell}}\right\}+\eta\right) \binom{n}{k-\ell}.\]
Applying Proposition~\ref{c} to $ H $, we get that  every triple of  vertices in $ V(H) $ contains two vertices which are $ (\beta, 1) $-reachable in $ H $. 
Without loss of generality, we suppose that $ \delta_{d}(H)\ge (1/3+\eta) \binom{n}{k-d} $ with $ d\ge2 $. 
So by Lemma~\ref{part}, we get a partition $\mathcal{P} $ of $ V(H) $ such that $ \mathcal{P}=\{V_1,U\} $ or $ \mathcal{P}=\{V_1,V_2,U\} $,  where $|V_{i}| \ge (\delta-\beta')n$ and $ V_{i} $ is $ (\beta, 2) $-closed in $ H $ for $ i\in[2] $, and $ 0\le |U| \le 2\delta n $.
Note that the case $ \mathcal{P}=\{V_1,U\} $ is indeed simpler. 
However, to unify the arguments, when $ \mathcal{P}=\{V_1,U\} $ we arbitrarily split $ V_{1} $ into two sets of equal size and by abusing the notation we call the resulting partition $ \mathcal{P}=\{V_1,V_2,U\} $. 
So we only need to deal with one case, where $ \mathcal{P}=\{V_1,V_2,U\} $, $|V_{i}| \ge (\delta-\beta')n$ and $ V_{i} $ is $ (\beta, 2) $-closed in $ H $ for $ i\in[2] $.
Suppose $ \mathcal{P}=\{V_1,V_2,U\} $ and let $ H':=H-U $.
\begin{fact}\label{a} 
There exists $ 1\le a\le k-1 $ such that $ (a,k-a)\in  I^{\mu}_{\mathcal{P}} (H') $.
\end{fact}
\begin{proof}
Let $ n':=n-|U| $ and $ n_1:=|V_1| $. 
Then $ n'\ge (1-2\delta) n $ and $\min\{n_1,n'-n_1\} \ge (\delta-\beta')n$.
Suppose  $ (a,k-a)\notin  I^{\mu}_{\mathcal{P}} (H') $ for all $ 1\le a\le k-1 $, that is, $ I^{\mu}_{\mathcal{P}} (H')\subseteq \{(0,k),(k,0)\}$.
So by $ d\ge2 $, we have  $ \sum_{S:  i_{\mathcal{P}}(S)=(1,d-1)}\deg_{H'}(S)\le  k\mu (n')^{k}\cdot k^{d}=k^{d+1}\mu (n')^{k} $.
On the other hand, by the minimum degree condition of $ H $, we have
\[
\begin{aligned}
\sum_{S:  i_{\mathcal{P}}(S)=(1,d-1)}\deg_{H'}(S)&\ge n_1\binom{n'-n_1}{d-1}\left(\delta_{d}(H)-2\delta n \binom{n}{k-d-1}\right)\\
&\ge (\delta-\beta')n\binom{(\delta-\beta')n}{d-1}\left((1/3+\eta) \binom{n}{k-d}-2\delta n \binom{n}{k-d-1}\right) \\
& >\delta^{d}/3\binom{n'}{k}> k^{d+1}\mu (n')^{k},
\end{aligned}
\] 
which is a contradiction.
So there exists $ 1\le a\le k-1 $ such that $ (a,k-a)\in  I^{\mu}_{\mathcal{P}} (H') $.	
\end{proof}

Suppose  $ (a,k-a)\in  I^{\mu}_{\mathcal{P}} (H') $ as in Fact~\ref{a}, where $ 1\le a\le k-1 $.
Let 
\[m:= 
\begin{cases}
	a & {a\le k/2}\\
	a-\ell+1 & {a> k/2}.
\end{cases} \]
Then $ 1\le m\le a $ and $ 0\le k-\ell-m\le k-a-1 $.
Let $ \mathcal{S} $ be the family of all $ (k-\ell) $-sets $ S $ with $ i_{\mathcal{P}}(S)\in \{(m,k-\ell-m),(m-1,k-\ell-m+1)\}$.
So for any $ S\in\mathcal{S} $, we have $ |S\cap V_1|\le a$ and $|S\cap V_2|\le k-a$.
The following claim says that every $ (k-\ell) $-set $ S\in\mathcal{S} $ has many absorbers. 
We postpone its proof later.
\begin{claim}\label{absorbers}
There exists $ b=(4k-2\ell-1)(k-\ell)+r $ for some $ r\le k^{4} $ such that the following holds. 
For any $ (k-\ell) $-set $ S\in \mathcal{S} $, $ H $ contains $ \beta^{k-\ell+1}n^{b}/2 $ $ S $-absorbers, each of which is a $ b $-tuple which spans a family of $ 2k- 2\ell + 1 $ vertex-disjoint $ \ell $-paths.
\end{claim} 

We select a family $\mathcal{T}$ of $b$-tuples at random independently from $ H $ by including each ordered $b$-set $T$  with probability $ \beta^{k-\ell+2}n^{1-b}$. 
For a fixed $ S $, let $ A_S $ be the set of all members of $\mathcal{T}$, 
each of which is an $ S $-absorber.
By Claim~\ref{absorbers}, $ \mathbb{E}[| A_S|]\ge\beta^{2k-2\ell+3}n/2 $.
Moreover,   \[\mathbb{E}[|\mathcal{T}|]=\beta^{k-\ell+2}n^{1-b}\binom{n}{b}b!\le \beta^{k-\ell+2}n,\]
\[ \mathbb{E}[|(T,T'):T,T'\ \text{in}\ \mathcal{T} \text{are intersecting} |]\le b^{2}n^{2b-1}(\beta^{k-\ell+2}n^{1-b})^{2}=\beta^{2k-2\ell+4}b^2n\](the corresponding unordered sets intersect).
By the Chernoff bound, Markov's inequality and the union bound, we can fix an outcome of our random selection of  $\mathcal{T}$ satisfying the following properties:
\begin{enumerate}
\item for every $(k-\ell)$-set  $ S\in \mathcal{S} $,  $\mathcal{T}$  contains  at least	$\beta^{2k-2\ell+3}n/4$ $ S $-absorbers;
\item  $ |\mathcal{T}|\le 2\beta^{k-\ell+2}n $;
\item there are at most $ 2\beta^{2k-2\ell+4}b^2n $ overlapping members of $\mathcal{T}$.
\end{enumerate}	
We delete one set from each overlapping pairs of members of $\mathcal{T}$. 
Also delete from $\mathcal{T}$ every member of $ \mathcal{T}$ which is  not an $ S $-absorber for any $ S\in \mathcal{S} $. 
Then we obtain a family $ \mathcal{F} $ of $b$-tuples such that  $ |\mathcal{F}|\le 2\beta^{k-\ell+2}n $, each $b$-tuple is an $ S $-absorber for some $ S\in \mathcal{S} $ and for each $ S\in \mathcal{S} $, $ \mathcal{F} $ contains at least $ \beta^{2k-2\ell+3}n/4-2\beta^{2k-2\ell+4}b^2n\ge \beta^{2k-2\ell+4}n $ $ S $-absorbers.
By Claim~\ref{absorbers}, each  $ S $-absorber is a $b$-tuple which spans a family  of  $ 2k- 2\ell + 1 $ vertex-disjoint $ \ell $-paths.
Denote all these $ \ell $-paths in $ \mathcal{F} $ by $ \{P_1,P_2,\dots,P_q\} $, where $ q\le2\beta^{k-\ell+2}n(2k-2\ell+1)\le \beta^{k-\ell+1}n$.

Now we shall use the minimum $d$-degree condition of $H$ to greedily construct disjoint edges to cover all vertices in $ U\setminus \mathcal{F} $ while avoiding the vertices of  the paths $P_{i}, 1 \le i \le q$.  
Since $\delta_{d}(H)\ge (1/3+\eta) \binom{n}{k-d}$, we have $\delta_{1}(H)\ge (1/3+\eta) \binom{n}{k-1}$.
Note that $ | U\setminus \mathcal{F}|\le|U|\le 2\delta n $ and $ |V(\mathcal{F})|\le2\beta^{k-\ell+2}nb  $.
For any set $ U'\subseteq V(H) $ with $ |U'|\le 2\beta^{k-\ell+2}nb +2k\delta n $, $\delta_{1}(H\setminus U')\ge \frac{1}{3} \binom{n}{k-1}$.
Thus, we find a matching $ M=\{P_{q+1}, ..., P_{q+h}\} $ of size $ h:=|U\setminus F|\le 2\delta n $ such that $ V(M) $ does not intersect any of the paths $ P_1,\dots, P_q $ and each edge of $ M $ contains exactly one vertex of $ U\setminus F $.

Let $P^{beg}_{i}$ and $P^{end}_{i}$ be the ordered ends of $P_{i}$ for $1\le i \le q+h$.
We see that  
$\delta_{\ell}(H\setminus U')\ge \left((1/2)^{\frac{k-\ell}{\ell}}+\frac{\eta}{2} \right)\binom{n-|U'|}{k-\ell}$ or  $\delta_{\ell+1}(H\setminus U')\ge \left(1/3+\frac{\eta}{2}\right)\binom{n-|U'|}{k-\ell-1}$ holds for any vertex set $ U' $ with $ |U'|\le \gamma n/2 $.
As $ (q+h)(b+8k^{5})\le(\beta^{k-\ell+1}n+2\delta n+1)(b+8k^{5})\le \gamma n/2$,
we can use Lemma~\ref{conlem} or Lemma~\ref{con lem}  to greedily connect	each ordered $\ell$-set $P^{end}_{i}$ to $P^{beg}_{i+1}$ by an $\ell$-path $P'_{i}$  with $ |V(P'_{i})|\le 8k^5 $, such that $P'_{i}$ intersects $P_{i}$ and $P_{i+1}$ only in the sets $P^{end}_{i}$ and $P^{beg}_{i+1}$ and does not intersect any other $P_{j}$ or any previously chosen $P'_{j}$. 
Having found these $\ell$-paths, we obtain an $\ell$-path $P'$ as $P_{1}P'_{1}P_{2}P'_{2}\dots P_{q+h-1}P'_{q+h-1}P_{q+h}$ with
$ |V(P')|\le \gamma n/2 $.
Note that for any $ S\in \mathcal{S} $, $ P' $ contains at least $ \beta^{2k-2\ell+4}n $ mutually disjoint $ S $-absorbers.
Thus $ P' $ can greedily absorb a vertex set $ W\in V(H)\setminus V(P') $, if  $ (k-\ell)||W| $, $|W|\le  (k-\ell)\beta^{2k-2\ell+4}n $ and there exist nonnegative integers $ x,y $ such that $ i_{\mathcal{P}}(W)=x(m,k-\ell-m)+y(m-1,k-\ell-m+1) $.

Let $p:= \lfloor \theta n \rfloor $. 
Take $4p$ mutually disjoint vertex sets $ S_1, S_2,\dots, S_{2p}, T_1, T_2,\dots, T_{2p} $ from $ V(H)\setminus V(P') $  with $ i_{\mathcal{P}}(S_i)=(m,k-\ell-m) $ and $ i_{\mathcal{P}}(T_i)=(m-1,k-\ell-m+1) $ for $i\in [2p] $.
We denote the union of  these   $ S_i $ and $T_i $, $ i\in[2p] $ by $ R_1 $. 
So $ R_1\subseteq V(H)\setminus V(P') $, $ |R_1|=4(k-\ell)p $ and  $ i_{\mathcal{P}}( R_1)=2p(m,k-\ell-m)+2p(m-1,k-\ell-m+1) $.
Thus $ P' $ can absorb $  R_1 $, that is, there exists an $\ell$-path $ P $ with the same ordered ends as $ P' $, where $ V(P)=V(P')\cup R_1 $.
Now we show  that $ P $ is the desired absorbing $\ell$-path.
Note that $ |V(P)|\le \gamma n $.
Fix any set $X\subseteq V(H)\setminus V (P)$ with $ |X|\le p $ and $ (k-\ell)\mid |X| $ as required by the lemma.
Suppose further that $ i_{\mathcal{P}}(X)=(t,s)=x(m,k-\ell-m)+y(m-1,k-\ell-m+1)  $. 
Then we have 
\[\begin{cases}
	x=t-\frac{(m-1)(t+s)}{k-\ell}\\
	y=\frac{(t+s)m}{k-\ell}-t.
\end{cases} \]
Since $ t+s=|X| $, $ (k-\ell)\mid |X| $ and $ m/(k-\ell)\le1 $, we get that $ x,y $ are integers and $ |x|,|y|\le|X|< 2p $.
Thus $ i_{\mathcal{P}}(X\cup R_1)=(x+2p)(m,k-\ell-m)+(y+2p)(m-1,k-\ell-m+1) $, where $ x+2p>0, y+2p>0 $ and $ |X\cup R_1|\le4(k-\ell)p+p \le(k-\ell)\beta^{2k-2\ell+4}n $.
So $ P' $ can absorb $ X\cup R_1 $, that is,  $ P$ can absorb $ X $.
\end{proof}	

To complete the proof of Lemma~\ref{abslem}, it remains to prove Claim~\ref{absorbers}.
\begin{proof}[Proof of Claim~\ref{absorbers}]
 Let  $ S:=\{v_1,v_2,\dots,v_{k-\ell}\}$ be a $ (k-\ell) $-set. 
Fix a $ k $-partite $ k $-graph $A(k,\ell)$ on $ [V_{1}^{A},V_{2}^{A},\dots,V_{k}^{A}] $ satisfying Proposition~\ref{Prop1} with $ |A(k,\ell)|=r\le k^4 $ and $V(A(k,\ell))=S'\cup X$, where $|S'| = k-\ell $.
Without loss of generality, suppose $ |S'\cap V_{i}^{A}|=1$ for $ i\in[k-\ell] $.
Let $ b:=(4k-2\ell-1)(k-\ell)+r $.
Since $ (a,k-a)\in  I^{\mu}_{\mathcal{P}} (H') $, the number of edges whose index vectors are $ (a,k-a) $ is at least $ \mu n^{k} $.
Note that $ t_1:=|S\cap V_1|\le a$ and $|S\cap V_2|\le k-a$  by $ S\in \mathcal{S} $.
Since $ A(k,\ell) $ is $ k $-partite, by the supersaturation result (see \cite{Onextremal}) on the subgraph of $ H $ that consists of all edges of index vector $ (a, k-a) $, $ H $ contains $ \beta n^{r} $ copies of $ A(k,\ell) $ each with $ V_{1}^{A},\dots,V_{t_1}^{A}\subseteq V_1 $ and $ V_{t_1+1}^{A},\dots,V_{k-\ell}^{A}\subseteq V_2 $.
For such a copy $ A $ of $A(k,\ell)$, we denote by $ S'_A $ as the set of $ k-\ell $ vertices given in Proposition~\ref{Prop1} (2). 
So $ i_{\mathcal{P}}(S'_{A})=i_{\mathcal{P}}(S) $ for each such $ A $.

Consider a copy of $A(k,\ell)$ in $ H $ which we denote by  $ A $. 
Note that each of $ V_1,V_2 $ is $ (\beta, 2) $-closed in $ H $.
Without loss of generality, suppose $ S'_{A}=\{w_1,w_2,\dots,w_{k-\ell}\} $ such that $ v_i , w_i $ are  $ (\beta, 2) $-reachable for $ i\in[k-\ell] $ by $ i_{\mathcal{P}}(S'_{A})=i_{\mathcal{P}}(S) $.
By the definition of reachability, for each $ i\in[k-\ell] $, there are at least $ \beta n^{4k-2\ell-1} $ $ (4k-2\ell-1) $-sets $ T_i $ such that there exist   $ \ell $-paths $ P_i^{1},P_i^{2},P_i^{3},P_i^{4} $  with $ V(P_i^{1}\cup P_i^{2})= T_i\cup\{v_i\}$ and $ V(P_i^{3}\cup P_i^{4})= T_i\cup\{w_i\} $, where $ P_i^{1} $  has the same ends as $P_i^{2} $, and  $ P_i^{3} $  has the same ends as $P_i^{4} $.
So there are at least $ \beta^{k-\ell+1}n^{b} $ choices for $ A\cup T_1\cup T_2 \cup \dots \cup T_{k-\ell} $ as an ordered set.
Among them, at most $ (k-\ell)n^{b-1} $ of them intersect $ S $ and at most $ b^2n^{b-1} $ of them contain repeated vertices.
Thus there are at least $ \beta^{k-\ell+1}n^{b}/2 $  $ b $-tuples avoiding $ S $ such that $ A, T_1, T_2 ,\dots, T_{k-\ell} $ are pairwise vertex-disjoint.

Now it remains to show that the $ b $-tuple corresponding to $ A\cup T_1\cup T_2 \cup \dots \cup T_{k-\ell} $ is an $ S $-absorber.
Firstly, $ T_i\cup\{w_i\}$, $ i\in [k-\ell] $ spans two vertex-disjoint $ \ell $-paths of length two, which together with the spanning $ \ell $-path in  $ A\setminus\{w_1,w_2,\dots,w_{k-\ell}\} $ form a family of $ 2k-2\ell+1 $ $ \ell $-paths which span $ V(A)\cup T_1\cup T_2 \cup \dots \cup T_{k-\ell} $.
Secondly, $H[T_i\cup\{v_i\}]$ forms two vertex-disjoint $ \ell $-paths of length two for $ i\in [k-\ell] $, which together with the spanning  $ \ell $-path in $ A $ gives a family of $ 2k-2\ell+1 $ $ \ell $-paths which span $ S\cup V(A)\cup T_1\cup T_2 \cup \dots \cup T_{k-\ell} $ and have the same ends as the family of $ \ell $-paths above. 
So  the $ b $-tuple corresponding to $ A\cup T_1\cup T_2 \cup \dots \cup T_{k-\ell} $ is an $ S $-absorber (cf.  Figure~\ref{f3}).
\end{proof}

\subsection{Proofs of Theorem~\ref{thmh2} and Theorem~\ref{thmh}}
We prove Theorem~\ref{thmh2} by following the  common approach of absorption (cf. \cite{MR2652102,2008An}).	
That is, we decompose  the proof in the usual way into the absorbing path lemma, the reservoir lemma, the connecting lemma and the path cover lemma.
We will use these lemmas to find an absorbing path and a reservoir firstly, and then we cover the majority of vertices by vertex-disjoint $ \ell $-paths. 
We connect up all these $ \ell $-paths to form an $ \ell $-cycle. 
Finally, we absorb the leftover vertices into the absorbing path, thereby completing a Hamilton $\ell$-cycle.

We need the following path cover lemma, which states that the vertex set of any sufficiently large $ k $-graph satisfying the minimum degree condition can be covered by a constant number of vertex-disjoint $ \ell $-paths with a small leftover. 
\begin{lemma}  [Path cover lemma \cite{large}, Lemma 3.1]\label{patlem1}
For all integers $k\ge3$, $1\leq \ell<k/2$ and  $1\leq d\leq k-1$, suppose  $1/n \ll 1/D\ll \varepsilon \ll\mu,1/k$. 
Let $H$ be a $k$-graph on $n$ vertices with $\delta_d(H)\ge (t(k,d,\ell)+\mu) \binom{n}{k-d}$.
Then there is a family of at most $D$ vertex-disjoint $\ell$-paths covering all but at most $4\varepsilon n$ vertices of $H$.
\end{lemma}	

We also need the reservoir lemma \cite{MR2652102} which guarantees the subhypergraph satisfying  the degree condition of the connecting lemma.	
\begin{lemma} [Reservoir lemma \cite{MR2652102}, Lemma 8.1] \label{res lem}  
Suppose that $k\ge2$,  $1 \le d \le k-1$ and  $1/n \ll \alpha ,\mu,1/k$. 
Let $H$ be a $k$-graph on $n$ vertices with $\delta_{d}(H) \ge \mu\binom{n}{k-d}$, and let $R$ be a subset of $V(H)$ of size $\alpha n$ chosen uniformly at random. 
Then with probability $1-o(1)$, $|N_{H}(S)\cap \binom{R}{k-d}|\ge \mu \binom{\alpha n}{k-d}-n^{k-d-1/3}$ holds for every $S \in  \binom{V(H)}{d}$.
\end{lemma}

Now we give the proof of Theorem~\ref{thmh2} by combining the results as outlined above.
\begin{proof}[Proof of Theorem~\ref{thmh2}]
Suppose we have the constants satisfying the following hierarchy
\[
1/n\ll 1/D \ll \varepsilon \ll \alpha \ll\theta\ll \beta  \ll \gamma \ll \eta \ll 1/k
\]
and  assume that $n\in (k-\ell)\mathbb{N}$.
Let $ t:= \max\{t(k,d,\ell),1/3\}$.
Suppose that $H$ is a $k$-graph on $n$ vertices such that $\delta_d(H)\ge (t+\eta) \binom{n}{k-d}$. 
Applying Lemma~\ref{abslem}, we obtain an $ \ell $-path $P_{0}$  with $|V(P_0)|\le\gamma n$, such that $P_0 $  can absorb any set $S\subseteq V(H)\setminus V (P_0)$ with $ |S|\le\theta n $, $ (k-\ell)\mid |S| $.

Next let $R$ be a set of $\alpha n$ vertices of $ V(H) $ chosen uniformly at random.
Applying Lemma~\ref{res lem} to $H$, we obtain that with probability $1-o(1)$, 
\[\left |N_{H}(S)\cap  \binom{R}{k-d}\right| \ge (t+\eta/2)\binom{\alpha n}{k-d}\]
for every $S \in \binom{V(H)}{d}$. 
Since $\mathbb{E}[|R \cap V(P_{0})|] = \alpha|V(P_{0})|$, by Markov's inequality, with probability at least $1/2$,  we have $|R \cap V(P_{0})| \le2\gamma\alpha n$. 
Then we fix a choice of $R$ which has the two properties above. 

Let $V':=V(H)\setminus (R \cup V(P_{0}))$. 
Note that $|R \cup V(P_{0})|\le \gamma n+\alpha n$. The induced subhypergraph $H':=H[V']$ satisfies 
\[\delta_d(H')\ge (t+\eta)\binom{ n}{k-d}-(\gamma n+\alpha n)\binom{ n}{k-d-1}\ge (t+\eta/2)\binom{|V'|}{k-d}.\]
Applying Lemma~\ref{patlem1} to $H'$, we obtain  vertex-disjoint $\ell$-paths $P_{1},\dots,P_{q}$ that together
cover all but at most $\varepsilon n$ vertices of $H'$, where $ q\le D $.
Denote by $X$  the set of uncovered vertices. Thus $|X| \le \varepsilon n$.
We  denote  the ordered ends of $P_{i}$ by $P^{beg}_{i}$ and $P^{end}_{i}$, $0 \le i \le q$. 
Let $ P^{beg}_{q+1}:=P^{beg}_{0} $.
For $0 \le i \le q$, we now find vertex-disjoint $\ell$-paths $P'_{i}$ by Lemma~\ref{con lem} to connect $P^{end}_{i}$ and  $P^{beg}_{i+1}$, which actually connects $P_{i}$ and $P_{i+1}$. Note that  $V(P'_{i})\subseteq  (R \setminus V(P_{0})) \cup P^{end}_{i} \cup P^{beg}_{i+1}$ and $|V(P'_{i})|\le 8k^{5}$.
More precisely, suppose that we have chosen such $\ell$-paths  $P'_{0},\dots,P'_{i-1}$. Let $R_{i}=\left( P^{end}_{i} \cup P^{beg}_{i+1}\cup R \setminus V(P_{0})\right)\setminus \bigcup_{j=0}^{i-1}V(P'_{j})$.  
Thus 
\[\delta_d(H[R_{i}])\ge ( t +\eta/2)\binom{\alpha n}{k-d}-(8k^{5}D+2\gamma\alpha n) \binom{\alpha n}{k-d-1}\ge t\binom{|V(R_{i})|}{k-d}\]
and thus we may apply Lemma~\ref{con lem} with $ t $ in place of $ \mu $ to find a desired $\ell$-path  $P'_{i}$.

Let $C: = P_{0}P'_{0}P_{1}P'_{1}\cdots P_{q}P'_{q}$ be the $\ell$-cycle we have obtained so far and let $R'' := V(H)\setminus V(C)$. 
Then indeed $ R''=X\cup \left(R \setminus \left(V(P_{0})\cup\bigcup_{0\le i\le q}V(P_i')\right)\right)  $ and in particular, $ |R''| \le (\alpha+\varepsilon)n \le\theta n$. 
Since $k-\ell$ divides both $ n $ and $|V(C)|$, we have $(k-\ell) \mid|R''|$. 
So we can utilize the absorbing property of $ P_0 $ to get an $\ell$-path $Q_{0}$ with  $V(Q_{0})= V(P_{0})\cup R''$ such that  $P_0$ and $Q_{0}$ have the same ordered ends, obtaining a Hamilton $\ell$-cycle  $C' := Q_{0}P'_{0}P_{1}P'_{1}\cdots P_{q}P'_{q}$ in $H$.
\end{proof}
Proof of Theorem~\ref{thmh} follows verbatim as the proof of Theorem~\ref{thmh2}, after replacing Lemma~\ref{con lem} with Lemma~\ref{conlem}. Thus we omit it.

\section{Acknowledgements}
We would like to express our gratitude to the anonymous reviewers for  their valuable comments that greatly improved the presentation of this paper.
Jie Han was supported by  Natural Science Foundation of China (12371341).
 Guanghui Wang was supported by  National Key R\&D Program of China (2020YFA0712400), Natural Science Foundation of China (12231018) and Young Taishan Scholars probgram of Shandong Province (201909001).

	
\bibliographystyle{abbrv}
\bibliography{52HCfinal}

\begin{thebibliography}{10}

\bibitem{Bastos2016Loose}
J.~D.~O. Bastos, G.~O. Mota, M.~Schacht, J.~Schnitzer, and F.~Schulenburg.
\newblock Loose {H}amiltonian cycles forced by large $(k-2)$-degree $-$
  approximate version.
\newblock {\em SIAM Journal on Discrete Mathematics}, 31(4):2328--2347, 2017.

\bibitem{J2017Loose}
J.~D.~O. Bastos, G.~O. Mota, M.~Schacht, J.~Schnitzer, and F.~Schulenburg.
\newblock Loose {H}amiltonian cycles forced by large $(k-2)$-degree $-$ sharp
  version.
\newblock {\em Contributions to Discrete Mathematics}, 13(2):88--100, 2018.

\bibitem{Bu2013Minimum}
E.~Bu\ss, H.~H{\`a}n, and M.~Schacht.
\newblock Minimum vertex degree conditions for loose {H}amilton cycles in
  3-uniform hypergraphs.
\newblock {\em Journal of Combinatorial Theory, Series B}, 114(6):658--678,
  2013.

\bibitem{CZYGRINOW2014TIGHT}
A.~Czygrinow and T.~Molla.
\newblock Tight codegree condition for the existence of loose {H}amilton cycles
  in 3-graphs.
\newblock {\em SIAM Journal on Discrete Mathematics}, 28(1):67--76, 2014.

\bibitem{1952Some}
G.~A. Dirac.
\newblock Some theorems on abstract graphs.
\newblock {\em Proceedings of the London Mathematical Society}, 2(1):69--81,
  1952.

\bibitem{Onextremal}
P.~Erd\H{o}s.
\newblock On extremal problems of graphs and generalized graphs.
\newblock {\em Israel Journal of Mathematics}, 2:183--190, 1964.

\bibitem{large}
L.~Gan, J.~Han, L.~Sun, and G.~Wang.
\newblock Large $ {Y}_{k,b} $-tilings and {H}amilton $ \ell $-cycles in
  $k$-uniform hypergraphs.
\newblock {\em Journal of Graph Theory}, 104(3/4):516--556, 2023.

\bibitem{han2021minimum}
H.~H{\`a}n, J.~Han, and Y.~Zhao.
\newblock Minimum degree thresholds for {H}amilton $(k/2)$-cycles in
  $k$-uniform hypergraphs.
\newblock {\em Journal of Combinatorial Theory, Series B}, 153:105--148, 2022.

\bibitem{H2010Dirac}
H.~H\`{a}n and M.~Schacht.
\newblock Dirac-type results for loose {H}amilton cycles in uniform
  hypergraphs.
\newblock {\em Journal of Combinatorial Theory, Series B}, 100(3):332--346,
  2010.

\bibitem{han2017decision}
J.~Han.
\newblock Decision problem for perfect matchings in dense $k$-uniform
  hypergraphs.
\newblock {\em Transactions of the American Mathematical Society},
  369(7):5197--5218, 2017.

\bibitem{Non-linear}
J.~Han, X.~Shu, and G.~Wang.
\newblock Non-linear {H}amilton cycles in linear quasi-random hypergraphs.
\newblock In {\em Proceedings of the 2021 {ACM}-{SIAM} {S}ymposium on
  {D}iscrete {A}lgorithms ({SODA})}, pages 74--88, 2021.

\bibitem{32}
J.~Han, L.~Sun, and G.~Wang.
\newblock Large $ {Y}_{3,2} $-tilings in 3-uniform hypergraphs.
\newblock {\em European Journal of Combinatorics}, 120, 2024.

\bibitem{Jie2015Minimum}
J.~Han and Y.~Zhao.
\newblock Minimum codegree threshold for {H}amilton $\ell$-cycles in
  $k$-uniform hypergraphs.
\newblock {\em Journal of Combinatorial Theory, Seires A}, 132:194--223, 2015.

\bibitem{Han2015Minimum}
J.~Han and Y.~Zhao.
\newblock Minimum vertex degree threshold for loose {H}amilton cycles in
  3-uniform hypergraphs.
\newblock {\em Journal of Combinatorial Theory, Series B}, 114:70--96, 2015.

\bibitem{2016Forbidding}
J.~Han and Y.~Zhao.
\newblock Forbidding {H}amilton cycles in uniform hypergraphs.
\newblock {\em Journal of Combinatorial Theory, Series A}, 143(6):107--115,
  2016.

\bibitem{2006Hamiltonian}
G.~Y. Katona and H.~A. Kierstead.
\newblock Hamiltonian chains in hypergraphs.
\newblock {\em Journal of Graph Theory}, 30(3):205--212, 1999.

\bibitem{2010Loose}
P.~Keevash, D.~K{\"u}hn, R.~Mycroft, and D.~Osthus.
\newblock Loose {H}amilton cycles in hypergraphs.
\newblock {\em Discrete Mathematics}, 311(7):544--559, 2010.

\bibitem{MR2652102}
D.~K\"{u}hn, R.~Mycroft, and D.~Osthus.
\newblock Hamilton {$\ell$}-cycles in uniform hypergraphs.
\newblock {\em Journal of Combinatorial Theory, Series A}, 117(7):910--927,
  2010.

\bibitem{2020Minimum}
R.~Lang and N.~Sanhueza-Matamala.
\newblock Minimum degree conditions for tight {H}amilton cycles.
\newblock {\em Journal of the London Mathematical Society}, 105(4):2249--2323,
  2020.

\bibitem{MR3338027}
A.~Lo and K.~Markstr\"{o}m.
\newblock {$F$}-factors in hypergraphs via absorption.
\newblock {\em Graphs and Combinatorics}, 31(3):679--712, 2015.

\bibitem{bollobas1981laszlo}
L.~Lov{\'a}sz.
\newblock Combinatorial problems and exercises.
\newblock {\em North-Holland Publishing Company}, 1993.

\bibitem{Markstr2011Perfect}
K.~Markstr{\"{o}}m and A.~Ruci{\'{n}}ski.
\newblock Perfect matchings (and {H}amilton cycles) in hypergraphs with large
  degrees.
\newblock {\em European Journal of Combinatorics}, 32(5):677--687, 2011.

\bibitem{Polcyn2020Minimum}
J.~Polcyn, C.~Reiher, V.~R{\"o}dl, A.~Ruci{\'{n}}ski, and B.~Sch\"ulke.
\newblock Minimum pair degree condition for tight {H}amiltonian cycles in
  $4$-uniform hypergraphs.
\newblock {\em Acta Mathematica Hungarica}, 161(2):647--699, 2020.

\bibitem{2020On}
J.~Polcyn, C.~Reiher, V.~R{\"o}dl, and B.~Sch\"ulke.
\newblock On {H}amiltonian cycles in hypergraphs with dense link graphs.
\newblock {\em Journal of Combinatorial Theory, Series B}, 150(1):17--75, 2021.

\bibitem{Reiher2019Minimum}
C.~Reiher, V.~R{\"o}dl, A.~Ruci{\'n}ski, M.~Schacht, and E.~Szemer{\'e}di.
\newblock Minimum vertex degree condition for tight {H}amiltonian cycles in
  3-uniform hypergraphs.
\newblock {\em Proceedings of the London Mathematical Society},
  119(2):409--439, 2019.

\bibitem{VOJTECH2006A}
V.~R{\"o}dl, A.~Ruci{\'{n}}ski, and E.~Szemer{\'{e}}di.
\newblock A {D}irac-type theorem for 3-uniform hypergraphs.
\newblock {\em Combinatorics, Probability and Computing}, 15(1,2):229--251,
  2006.

\bibitem{2008An}
V.~R{\"o}dl, A.~Ruci{\'{n}}ski, and E.~Szemer{\'{e}}di.
\newblock An approximate {D}irac-type theorem for $k$-uniform hypergraphs.
\newblock {\em Combinatorica}, 28(2):229--260, 2008.

\bibitem{2011Dirac}
V.~R{\"o}dl, A.~Ruci{\'{n}}ski, and E.~Szemer{\'{e}}di.
\newblock Dirac-type conditions for {H}amiltonian paths and cycles in 3-uniform
  hypergraphs.
\newblock {\em Advances in Mathematics}, 227(3):1225--1299, 2011.

\end{thebibliography}

\begin{appendix}
\section{Proof of Theorem~\ref{52}}
In this section we prove Theorem~\ref{52}.
Considering the 5-graph $ H_{5,2} $ in Section 1 and Theorem~\ref{thmh}, we have 
\[  \frac{91}{216}\le h^{2}_{2}(5)\le \max\left\{(1/2)^{\frac{3}{2}},t(5,2,2),1/3\right\}.\] 
So it suffices to prove $t(5,2,2) \le \frac{91}{216} $.
We will prove the following theorem to determine $  t(5,2,2) $.
\begin{thm} \label{yth}
Let $ 0<1/n\ll\varepsilon\ll \eta $.	
Let $H$ be a 5-graph on $n$ vertices with \[\delta_{2}(H)\ge \left(\frac{91}{216}+\eta\right) \binom{n}{3}.\] 
Then $H$ contains a $Y_{5,4}$-tiling covering all but at most $\varepsilon n$ vertices.
In particular, $ t(5,2,2)= \frac{91}{216} $.
\end{thm}	

For $ p>0 $, fix two $k$-graphs,  $F$ of order $ p $, and $H$. 
Let $\mathcal F_{F,H} \subseteq \binom {V(H)}{p}$ be the family of $p$-sets in $V(H)$ that span a copy of $F$.
A \emph{fractional $F$-tiling} in $H$ is a function $\omega: \mathcal F_{F,H} \to [0,1]$ such that for each $v\in V(H)$ we have $\sum_{v\in e\in  \mathcal F_{F,H}}\omega(e) \le 1$. 
Then $\sum_{e\in  \mathcal F_{F,H}}\omega(e)$ is the \emph{size} of $\omega$. 
Such a fractional $F$-tiling is called perfect if it has size $n/p$.
Let $ 0\le d \le k-1 $ and $ 0\le s\le n/p $.
We denote by $f_d^s(F,n)$  the minimum $m$ so that every $n$-vertex $k$-graph $H$ with $\delta_d(H)\ge m$ has a fractional $F$-tiling of size $s$. 
In particular,  $\delta_0(H)=e(H)$.
Gan, Han, Sun and Wang \cite{large} proved the following two lemmas for fractional $Y_{k,b}$-tilings.
\begin{lemma}[\cite{large}, Lemma 7.2] \label{ftot}
Suppose $ 0<1/n\ll\varepsilon\ll \eta $. Let $ H $ be a $5 $-graph on $ n $ vertices with $\delta_{2}(H)\ge f_2^{n/6}(Y_{5, 4}, n)+\eta\binom{n}{3}$. 
Then $ H $ contains a $Y_{5,4}$-tiling covering all but at most $\varepsilon n$ vertices.
\end{lemma}
\begin{lemma}[\cite{large}, Lemma 7.1] \label{tran}
For $n \ge 5$, we have $f_2^{n/6}(Y_{5, 4}, n) \le f_0^{n/6}(Y_{3,2}, n-2)$.
\end{lemma}	

By Lemma~\ref{ftot}, to prove Theorem~\ref{yth}, it suffices to find an almost perfect fractional $Y_{5,4}$-tiling under the minimum 2-degree condition. 
Then by Lemma~\ref{tran}, we transform this problem to finding a large $Y_{3,2}$-tiling under density condition, so that we can apply the following theorem in Han, Sun and Wang \cite{32}.
\begin{thm}\label{thm1}
For every $ \alpha, \gamma\in (0,1/4) $ there exists an integer $ n_{0} $ such that the following holds for every integer $n\ge n_0$. 
Let $H$ be a 3-graph on $n$ vertices such that
\[\label{ec} e(H)\ge \max \left \{ \binom{4\alpha n}{3}, \binom{n}{3}-\binom{n-\alpha n}{3} \right \}+\gamma n^3.\]
Then $H$ contains a $Y_{3,2}$-tiling covering more than $ 4\alpha n$ vertices.		
\end{thm}
\begin{proof}	[Proof of Theorem~\ref{yth}]
Suppose that we have constants such that  $0<1/n\ll  \varepsilon \ll \eta $.
Let $H$ be a $5$-graph on $n$ vertices such that $\delta_2(H)\ge (\frac{91}{216}+\eta) \binom{n}{3}$.  
By Lemma~\ref{tran} and  Theorem~\ref{thm1}  with $ \eps $ in place of $ \gamma $ and $ \alpha=\frac{n}{6(n-2)} $, we have 
\[\begin{aligned} 
	f_2^{n/6}(Y_{5, 4}, n) &\le f_0^{n/6}(Y_{3,2}, n-2)\\&\le \max\left\{\binom{\frac{2n^2}{3(n-2)}}{3},\binom{n-2}{3}-\binom{(n-2)\left(1-\frac{n}{6(n-2)}\right)}{3}\right\}+\varepsilon (n-2)^3\\
	&\le \left(\frac{91}{216}+\frac{\eta}{2}\right) \binom{n}{3}.
\end{aligned}\]
Thus we have $ \delta_{2}(H)\ge (\frac{91}{216} +\eta)\binom{n}{3}\ge	f_2^{n/6}(Y_{5, 4}, n)+\frac{\eta}{2} \binom{n}{3}. $ 
Applying   Lemma~\ref{ftot} with $ \eta/2 $ in place of $ \eta $, we obtain a $Y_{5,4}$-tiling covering all but at most $\varepsilon n$ vertices.                      

In particular, by the definition of $t(5,2,2)$,  we get $ t(5,2,2)\le \frac{91}{216}$.
By considering the 5-graph $ H_{5,2} $ in Section 1 and Theorem~\ref{thmh}, we have 
$   \frac{91}{216}\le h^{2}_{2}(5)\le t(5,2,2) $.
Thus $ t(5,2,2)= \frac{91}{216}$.
\end{proof}

\end{appendix}	
\end{document}